%% file: main.tex
\documentclass[letterpaper, 10 pt, conference]{ieeeconf}  

\IEEEoverridecommandlockouts                              
\usepackage{booktabs} 
\usepackage{amsfonts}
\usepackage{dblfloatfix}
\usepackage{mathrsfs} 
\usepackage{boldline}
\usepackage[table]{xcolor}
\usepackage{amsmath} 
\usepackage{psfrag} 
\usepackage{multirow} 
\usepackage{arydshln}
\usepackage{enumerate}
\usepackage{bbm}
\usepackage{graphicx}
\usepackage{color}
\usepackage{setspace}
\usepackage{mathtools}
\let\labelindent\relax
\usepackage{enumitem}
\def\mathclap#1{\text{\hbox to 0pt{\hss$\mathsurround=0pt#1$\hss}}}

\newtheorem{definition}{Definition}

\newtheorem{corollary}{Corollary}
\newtheorem{lemma}{Lemma}
\newtheorem{theorem}{Theorem}

\newcommand{\Real}{\mathbb{R}}
\newcommand{\Complex}{\mathbb{C}}
\newcommand{\Zahlen}{\mathbb{Z}}

\newcommand{\T}{\mathsf{T}}
\newcommand{\Hn}{\mathsf{H}}

\newcommand{\rank}{{\rm rank}}
\newcommand{\spanv}{{\rm span}}

\newcommand{\nul}{{\rm null}}
\newcommand{\tr}{{\rm tr}}

\newcommand{\base}{{\bf{e}}}
\newcommand{\Base}{{\bf{E}}}

\newcommand{\zero}{\mathbf{0}}

\newcommand{\iu}{\mathrm{i}}
\newcommand{\R}{\mathrm{Re}}
\newcommand{\I}{\mathrm{Im}}

\newcommand{\Graph}{\mathcal{G}}
\newcommand{\Vertex}{\mathcal{V}}
\newcommand{\Edge}{\mathcal{E}}
\newcommand{\M}{\mathcal{M}}
\newcommand{\set}{\mathcal{S}}
\newcommand{\neighbor}{\mathcal{N}}

\newcommand{\objset}{\mathcal{S}_{\rm o}}
\newcommand{\conset}{\mathcal{S}_{\rm c}}
\newcommand{\consetcol}{\tilde{\conset}}

\newcommand{\Y}{{\bf Y}}
\newcommand{\V}{{\bf V}}
\newcommand{\VV}{{\bf v}}
\newcommand{\s}{{s}}
\newcommand{\MatR}{{\bf \Phi}}
\newcommand{\MatI}{{\bf \Psi}}
\newcommand{\W}{{\bf W}}
\newcommand{\C}{{\bf C}}
\newcommand{\A}{{\bf A}}
\newcommand{\X}{{\bf X}}
\newcommand{\Q}{{\bf Q}}
\newcommand{\x}{{\bf x}}
\newcommand{\y}{{\bf y}}
\newcommand{\yy}{\tilde{{\bf y}}}

\newcommand{\Wlim}{\W_{\rm lim}}
\newcommand{\Wopt}{\W_{\rm opt}}

\newcommand{\eps}{\varepsilon}

\newcommand{\vu}{\overline{v}_j^\phi}
\newcommand{\vl}{\underline{v}_j^\phi}
\newcommand{\pu}{\overline{p}_j^\phi}
\newcommand{\pl}{\underline{p}_j^\phi}
\newcommand{\qu}{\overline{q}_j^\phi}
\newcommand{\ql}{\underline{q}_j^\phi}

\newcommand{\dpur}{\overline{\mu}}
\newcommand{\dplr}{\underline{\mu}}
\newcommand{\dqur}{\overline{\eta}}
\newcommand{\dqlr}{\underline{\eta}}
\newcommand{\dvu}{\overline{\lambda}_{j}^{\phi}}
\newcommand{\dvl}{\underline{\lambda}_{j}^{\phi}}
\newcommand{\dpu}{\overline{\mu}_{j}^{\phi}}
\newcommand{\dpl}{\underline{\mu}_{j}^{\phi}}
\newcommand{\dqu}{\overline{\eta}_{j}^{\phi}}
\newcommand{\dql}{\underline{\eta}_{j}^{\phi}}
\newcommand{\dslack}{\kappa}

\newcommand{\coef}{\sigma_{ab}^{\phi,:}}
\newcommand{\coefc}{\sigma_{ab}^{:,\phi}}

\newcommand{\indp}{f_p}
\newcommand{\indq}{f_q}

\newcommand{\getnz}{\Omega}
\newcommand{\eigv}{\varrho}
\newcommand{\eigvec}{{\bf u}}
\newcommand{\U}{{\bf U}}
\newcommand{\z}{{\bf z}}
\newcommand{\keyvec}{\tilde{\bf u}}
\newcommand{\minvec}{\hat{\bf u}}

\newcommand{\up}{\!\!u\!\!}
\newcommand{\lo}{\!\!l\!\!}

\usepackage{comment}
\usepackage{ifthen}
\newboolean{showcomments}
\setboolean{showcomments}{true}
\newcommand{\fengyu}[1]{\ifthenelse{\boolean{showcomments}}
{ \textcolor{red}{(SL:  #1)}}{}}
\newcommand{\yue}[1]{\ifthenelse{\boolean{showcomments}}
{ \textcolor{green}{(SL:  #1)}}{}}
\newcommand{\slow}[1]{\ifthenelse{\boolean{showcomments}}
{ \textcolor{blue}{(SL:  #1)}}{}}

\title{\LARGE \bf
Sufficient Conditions for Exact Semidefinite Relaxation of Optimal Power Flow in Unbalanced Multiphase Radial Networks}

\author{Fengyu Zhou, Yue Chen, and Steven H. Low
\thanks{Fengyu Zhou is with the Department of Electrical Engineering, California Institute of Technology, Pasadena, CA, 91125. Email: {\tt\small f.zhou@caltech.edu}}%
\thanks{Yue Chen is with the State Key Laboratory of Power Systems, Department of Electrical Engineering, Tsinghua University, 100084 Beijing, China. Email: {\tt\small yue-chen15@mails.tsinghua.edu.cn}}
\thanks{Steven H. Low is with the Department of Electrical Engineering and the Department of Computing and Mathematical Sciences, California Institute of Technology, Pasadena, CA, 91125. Email: {\tt\small slow@caltech.edu}}%
}

\begin{document}

\maketitle
\thispagestyle{empty}
\pagestyle{empty}

\begin{abstract}
This paper proves that in an unbalanced multiphase network with a tree topology, the semidefinite programming relaxation of optimal power flow problems is exact 
when critical buses are not adjacent to each other. 
Here a critical bus either contributes directly to the cost function or is where an injection constraint
is tight at optimality.
Our result generalizes a sufficient condition for exact relaxation in single-phase tree networks
to tree networks with arbitrary number of phases.
\end{abstract}


\input{Introduction}
\input{Model}
\input{Perturbation}
\input{Condition}
\input{Proof}
\input{Simulation}





\bibliographystyle{IEEEtran}
\bibliography{my-bibliography}

\end{document}

%% file: Introduction.tex
\section{Introduction}
Optimal power flow (OPF) is a mathematical program that minimizes disutility subject to physical laws and other constraints \cite{momoh1999review}. 
It was first proposed in \cite{carpentier1962contribution} and there is a vast literatures on a large number of different solution methods.
In general, the OPF problem under alternating current (AC) model is both non-convex and NP-hard \cite{BienstockVerma2015, lehmann2016ac}.
There is thus a strong interest in studying its convexification or approximation; see e.g. a recent survey 
in \cite{molzahn2019survey} on relaxations and approximations of OPF.
Using semidefinite programming (SDP) to relax the non-convex constraints was first proposed in \cite{jabr2006radial, Bai2008}, and turns out to have good performance in many testcases \cite{lavaei2012zero, lesieutre2011examining}.
Many papers have proposed sufficient conditions under which the SDP relaxation is exact for a single-phase radial network (i.e. network with a tree topology)
or its single-phase equivalent of a balanced three-phase network, e.g. \cite{farivar2013branch, zhang2013geometry, bose2015quadratically, sojoudi2014exactness, gan2015exact}.

Most radial networks are however unbalanced multiphase, e.g., \cite{srinivas2000distribution, Kersting2002}.
SDP relaxation has recently been applied to unbalanced multiphase radial networks 
\cite{dall2013distributed, gan2014convex, zhao2017convex, zamzam2018beyond}.
Simulation results in these papers suggest that SDP relaxation is often exact even though no sufficient condition
for exact relaxation is known to the best of our knowledge.
Indeed, it has been observed in \cite{Berg1967, Laughton1968, Chen1991}
that a multiphase unbalanced network has an equivalent single-phase circuit
model where each bus-phase pair in the multiphase network is identified with a single
bus in the equivalent model.   The single-phase equivalent model is then a meshed network and therefore existing guarantees
on exact SDP relaxation are not applicable.
Most distribution systems are unbalanced multiphase networks \cite{christakou2013efficient}
and hence the performance of SDP relaxation of OPF on these networks is important.

In this paper, we generalize the sufficient condition for single-phase network proposed in \cite{bose2015quadratically} to the multiphase setting.
It is shown that when the critical buses or bottleneck buses in a network are non-adjacent, then the SDP relaxation is exact.
We prove in this paper the exactness results when the SDP has a unique solution, 
and state the result for the case of multiple solutions without proof.

%% file: Model.tex
\section{System Model}
\subsection{Network Structure}
We use a similar model as in \cite{gan2014convex, zhao2017convex}.We assume that all buses have the same number of phases and 
all generations and loads are Wye-connected.
Let the underlying simple undirected graph be $\Graph=(\Vertex,\Edge)$ where $\Vertex=\{0,1,\dots,n-1\}$ denotes the set of buses and $\Edge$  the set of edges. Throughout the paper, we will use (graph, vertex, edge) and (power network, bus, line) interchangeably. 
Without loss of generality, we let bus $0$ be the slack bus where the voltage is specified.
Assume all buses have $m$ phases for $m\in\Zahlen^+$. We will use $(j,k)$ and $j\sim k$ interchangeably to denote an edge connecting bus $j$ and $k$.
Consider an $m$-phase line $(j,k)$ characterized by the admittance matrix
$y_{jk} \in \Complex^{m\times m}$, we assume $y_{jk} $ is invertible. The admittance matrix $\Y\in\Complex^{mn\times mn}$ for the entire network can be divided into $n\times n$ number of $m\times m$ block matrices.
Let $\Y_{jk}\in\Complex^{m\times m}$ denote the block matrix corresponding to the admittance between bus $j$ and $k$, then we have
\begin{align*}
\Y_{jj}&=\sum\limits_{k:j\sim k} y_{jk},~j\in\Vertex\\
\Y_{jk}&=\left\{
\begin{array}{ll}
-y_{jk}&,j\sim k\\
0&,j\not\sim k
\end{array}
\right..
\end{align*}

For each bus $j$, let the voltages of all $m$ phases at bus $j$ be the vector $\V_j\in\Complex^{m}$. We use $\V_j^{\phi}$ for $\phi\in\M:=\{1,2,\dots,m\}$ to indicate the voltage for phase $\phi$.
Let $\V=[\V_0^\T,\V_1^\T,\dots,\V_{n-1}^\T]^\T$ be the voltage vector for the entire network. Similarly, we use $\s_j^\phi$ to denote the bus injection for phase $\phi$ at bus $j$.
Let $\base_j^\phi\in\Real^{mn}$ be the base vector which has $1$ at the $(jm+\phi)$\textsuperscript{th} entry and $0$ elsewhere. Let $\Base_j^\phi=\base_j^\phi(\base_j^\phi)^\T$, then we define
\begin{align*}
Y_j^\phi := \Base_j^\phi\Y\in\Complex^{mn\times mn}
\end{align*}
and
\begin{align*}
\MatR_j^\phi &:= \frac{1}{2}\big((Y_j^\phi)^\Hn + Y_j^\phi\big)\\
\MatI_j^\phi &:= \frac{1}{2\iu}\big((Y_j^\phi)^\Hn - Y_j^\phi\big).
\end{align*}
Both $\MatR$ and $\MatI$ are Hermitian matrices. The relationship between bus voltages and injections can be expressed as
\begin{align}\label{eq:s2V}
\nonumber
\R(\s_j^\phi) &= \V^\Hn \MatR_j^\phi \V,\\
\I(\s_j^\phi) &= \V^\Hn \MatI_j^\phi \V. 
\end{align}

\subsection{Optimal Power Flow}
Optimal power flow problems  minimize certain cost functions subject to constraints involving voltages and injections. 
Here we consider problems that take the linear combination of bus injections as the cost function and are subject to operational constraints for both voltage magnitudes and real/reactive injections.
For  problems with nonlinear cost functions, see Section \ref{sec:discuss}.
Suppose the bounds $\underline{V}$ and $\overline{V}$ for the voltage magnitudes are always positive and finite, but the bounds for real/reactive injections can be 
$\pm\infty$ if there are no such constraints.
\begin{subequations}
\begin{eqnarray}
\underset{\V,\s}{\text{minimize}}  && \sum\limits_{j,\phi}c_{j,{\rm re}}^\phi \R(s_j^\phi) + c_{j,{\rm im}}^\phi \I(s_j^\phi)
\label{eq:opf1.a}\\
\text{\quad subject to}& & \eqref{eq:s2V}
\label{eq:opf1.b}\\
& &\underline{V}_j^\phi \leq |\V_j^\phi| \leq \overline{V}_j^\phi, \quad\forall j,\phi
\label{eq:opf1.c}\\
& &\underline{p}_j^\phi \leq \R(\s_j^\phi) \leq \overline{p}_j^\phi, \quad\forall j,\phi
\label{eq:opf1.d}\\
& &\underline{q}_j^\phi \leq \I(\s_j^\phi) \leq \overline{q}_j^\phi, \quad\forall j,\phi
\label{eq:opf1.e}\\
& &\V_0=\V_{\rm ref}
\label{eq:opf1.f}
\end{eqnarray}
\label{eq:opf1}
\end{subequations}
Here, $\V_{\rm ref}\in\Complex^m$ denotes the reference voltage for $m$ phases at the slack bus. 
Substituting the decision variables $\s$ and $\V$ with $\W:=\V\V^\Hn$, the following equivalent formulation of \eqref{eq:opf1} is obtained.
\begin{subequations}
\begin{eqnarray}
\underset{\W\geq 0}{\text{minimize}}  && \tr(\C_0\W)
\label{eq:opf2.a}\\
\text{\quad subject to}& &\vl \leq \tr(\Base_j^\phi\W) \leq \vu, \quad\forall j,\phi
\label{eq:opf2.b}\\
& &\pl \leq \tr(\MatR_j^\phi\W) \leq \pu, \quad\forall j,\phi
\label{eq:opf2.c}\\
& &\ql \leq \tr(\MatI_j^\phi\W) \leq \qu, \quad\forall j,\phi
\label{eq:opf2.d}\\
& &[\W]_{00}=\VV_{\rm ref}
\label{eq:opf2.e}\\
& &\rank(\W)=1.
\label{eq:opf2.f}
\end{eqnarray}
\label{eq:opf2}
\end{subequations}
Here, $\underline{v}_j^\phi = |\underline{V}_j^\phi|^2$, $\overline{v}_j^\phi = |\overline{V}_j^\phi|^2$, $\VV_{\rm ref} = \V_{\rm ref}\V_{\rm ref}^\Hn$, and $[\W]_{00}$ stands for the upper left $m\times m$ submatrix of $\W$.
The cost matrix $\C_0=\sum_{j,\phi}c_{j,{\rm re}}^\phi \MatR_j^\phi + c_{j,{\rm im}}^\phi \MatI_j^\phi$. 
Dropping the rank-1 constraint in \eqref{eq:opf2.f} yields the semidefinite relaxation.
\begin{subequations}
\begin{eqnarray}
\underset{\W\geq 0}{\text{minimize}}  && \tr(\C_0\W)
\label{eq:opf3.a}\\
\text{\quad subject to}& &\eqref{eq:opf2.b}-\eqref{eq:opf2.e}.
\label{eq:opf3.b}
\end{eqnarray}
\label{eq:opf3}
\end{subequations}
We use the following exactness definition.
\begin{definition}\label{df:exactness}
A relaxation problem \eqref{eq:opf3} is {\it exact} if at least one of its optimal solutions
$\W^*$ is of rank 1.
\end{definition}
Given a rank-1 solution $\W^*$ of \eqref{eq:opf3}, a $\V^*$ can be uniquely determined,
which is feasible, and hence optimal, for \eqref{eq:opf2}.
%

We first make the assumption that \eqref{eq:opf3} has a unique optimal solution.
In Section \ref{sec:discuss}, we discuss the case when multiple optimal solutions exist.

%% file: Perturbation.tex
\section{Perturbation Analysis}
We first study a perturbed version of \eqref{eq:opf3}. 
\subsection{Perturbed Problem}
Fix a nonzero Hermitian matrix $\C_1$, and consider the following perturbed problem for $\eps\geq0$.
\begin{subequations}
\begin{eqnarray}
\underset{\W\geq 0}{\text{minimize}}  && \tr((\C_0+\eps\C_1)\W)
\label{eq:opf_perturb.a}\\
\text{\quad subject to}& &\eqref{eq:opf2.b}-\eqref{eq:opf2.e}.
\label{eq:opf_perturb.b}
\end{eqnarray}
\label{eq:opf_perturb}
\end{subequations}
We say that \eqref{eq:opf_perturb} is {\it exact} if one of its optimal solution is of rank 1.

\begin{lemma}\label{thm:perturbation}
For any nonzero $\C_1$, if there exists a sequence $\{\eps_{l}\}_{l=1}^{\infty}$ with $\lim_{l\to\infty}\eps_l=0$ 
such that \eqref{eq:opf_perturb} is exact for all $\eps_l$, then \eqref{eq:opf3} is exact.
\end{lemma}
\begin{proof}
Suppose the rank-1 optimal solution to \eqref{eq:opf_perturb} for $\eps_l$ is $\W_l$.
If the rank-1 optimal solution is non-unique, then pick any one as $\W_l$.
As all the $\vu$ are finite, we assume they are upper bounded by a constant $\alpha$.
Hence the constraint \eqref{eq:opf2.b} implies all the diagonal elements of $\W$ are upper bounded by $\alpha$.
Since $\W$ is positive semidefinite, the norms of all their entries can be upper bounded by $\alpha$ as well.
Consider the set
\begin{align}
\set=\{\W\geq0:\eqref{eq:opf2.b}-\eqref{eq:opf2.f}\}.
\end{align}
The set $\{\W:\rank(\W)\leq1\}$ is closed \cite{horn1990matrix} and all other constraints \eqref{eq:opf2.b}-\eqref{eq:opf2.e} also prescribe closed sets.
The only $\W$ with rank $0$ is the zero matrix which violates \eqref{eq:opf2.e} and is thus infeasible.
Further, we have shown that for any $\W\in\set$, its max norm must be upper bounded by $\alpha$,
so $\set$ is compact.
The infinite set $\{\W_{l}\}_{l=1}^{\infty}$ is a subset in $\set$ and hence has a limit point $\Wlim\in\set$ \cite{rudin1964principles}.
For any $\eps_l$, \eqref{eq:opf_perturb} has the same feasible set as \eqref{eq:opf3}, and hence the rank-1 matrix 
$\Wlim$ is also feasible for \eqref{eq:opf3}.
Next we show that $\Wlim$ is also an optimal point for \eqref{eq:opf3}.

If there exists another feasible $\Wopt\neq\Wlim$ such that $\tr(\C_0\Wlim)-\tr(\C_0\Wopt)=\nu>0$. Clearly $\forall\W$ feasible for \eqref{eq:opf3}, $|\tr(\C_1\W)|\leq m^2n^2\|\C_1\|_{\infty}\|\W\|_{\infty}\leq m^2n^2\alpha\|\C_1\|_{\infty}$.
For sufficiently large $l$ such that
\begin{align*}
\eps_l &< \frac{\nu}{4m^2n^2\alpha\|\C_1\|_{\infty}}\\
\|\W_l-\Wlim\|_{\infty}&<\frac{\nu}{4m^2n^2\|\C_0\|_\infty},
\end{align*}
we have
\begin{subequations}
\begin{eqnarray}
\tr(\C_0(\W_l-\Wlim)) &\geq&-\frac{\nu}{4} \label{eq:ineq1.a}\\
\tr(\eps_l\C_1\W_l) &\geq& -\frac{\nu}{4} \label{eq:ineq1.b}\\
\tr(\C_0\Wlim) &=& \tr(\C_0\Wopt)+\nu \label{eq:ineq1.c}\\
\frac{\nu}{4} &\geq& \tr(\eps_l\C_1\Wopt) \label{eq:ineq1.d}
\end{eqnarray}
\label{eq:ineq1}
\end{subequations}
Summing up \eqref{eq:ineq1.a}-\eqref{eq:ineq1.d} gives
\begin{align*}
&\tr((\C_0+\eps_l\C_1)\W_l)>\tr((\C_0+\eps_l\C_1)\Wopt).
\end{align*}
contradicting the optimality of $\W_l$ for $\eps_l$.
\end{proof}

\subsection{Duality}
The dual problem of \eqref{eq:opf_perturb} is as follows.
\begin{subequations}
\begin{eqnarray}
\nonumber
\underset{\substack{\dvu,\dvl,\dpu,\dpl,\\\dqu,\dql,\dslack}}{\text{maximize}}  && -\sum\limits_{j,\phi} (\dvu\vu-\dvl\vl+\dpu\pu-\dpl\pl+ \nonumber\\[-18pt]
&&\quad\quad\quad \dqu\qu-\dql\ql+\tr(\dslack\VV_{\rm ref}))\\
\label{eq:opf_dual.a}
\text{\quad subject to}& &\dvu,\dvl,\dpu,\dpl,\dqu,\dql \geq 0
\label{eq:opf_dual.b}\\
&&\A(\eps)\geq 0.
\label{eq:opf_dual.c}
\end{eqnarray}
\label{eq:opf_dual}
\end{subequations}
Dual variables $(\dvu,\dvl),(\dpu,\dpl),(\dqu,\dql)$ and  $\dslack$ correspond to \eqref{eq:opf2.b}-\eqref{eq:opf2.e} in \eqref{eq:opf_perturb.b}, respectively.
Specifically $\dslack\in\Complex^{m\times m}$ is Hermitian but not necessarily semidefinite positive.
Matrix $\A(\eps)$ denotes
\begin{align}\label{eq:A}
\nonumber
\A(\eps):=&\sum\limits_{j,\phi}(\dvu-\dvl)\Base_j^\phi+(\dpu-\dpl)\MatR_j^\phi+(\dqu-\dql)\MatI_j^\phi\\
&+\C_0+\eps\C_1+\Pi(\dslack)
\end{align}
and $\Pi(\dslack)$ is an $mn\times mn$ matrix whose upper left $m\times m$ block is $\dslack$ and other elements are $0$.
Note that the upper and lower bounds in \eqref{eq:opf2.c} and \eqref{eq:opf2.d} could take values of $\pm\infty$.
However, since the feasible set prescribed by \eqref{eq:opf_perturb.b} is compact, the actual values of 
$\MatR_j^\phi\W$ and $\MatI_j^\phi\W$ are always finite and hence the dual variables associated with such
constraints will be $0$. These constraints can be removed from \eqref{eq:opf_perturb} and \eqref{eq:opf_dual}.
We will use $\dvu(\eps)$, $\dvl(\eps)$ and so on to denote the Lagrange multipliers for $\eps$. 
Clearly, when $\eps=0$, \eqref{eq:opf_dual} is the dual problem of \eqref{eq:opf3} with 
$\dvu(0)$, $\dvl(0)$ and so on as the Lagrange multipliers.
If the value of $\eps$ is clear in the context, we might denote them simply as $\dvu$, $\dvl$ and so on for convenience.
Let $\A^*(\eps)$ be the dual matrix when dual variables are evaluated at a KKT point.

%% file: Condition.tex
\section{Sufficient Conditions}
The first condition we assume is:
\begin{enumerate}[label=A\arabic*:]
\item Problem \eqref{eq:opf3} is strictly feasible, 
{i.e., there exists a feasible point such that strict inequality holds in \emph{all}
inequality constraints in \eqref{eq:opf2.b}-\eqref{eq:opf2.e}. }
\end{enumerate}

Then the Slater's condition is satisfied for both \eqref{eq:opf3} and \eqref{eq:opf_perturb} as they share the same feasible set, and the strong duality between \eqref{eq:opf_perturb} and \eqref{eq:opf_dual} holds.
The KKT condition is necessary and sufficient optimality condition for the primal \eqref{eq:opf_perturb} and the dual \eqref{eq:opf_dual}
problem.
In this section, $\W^*$ refers to the unique solution of \eqref{eq:opf3}.

\subsection{Notations}
The following notations and definitions will be used throughout the rest of the paper.

For each bus-phase pair $(j,\phi)$, we define
\begin{align*}
\indp(j,\phi):=\left\{
\begin{array}{ll}
0  ,&\tr(\MatR_j^\phi\W^*) \not\in\{ \pu,\pl\}\\
1  ,&\tr(\MatR_j^\phi\W^*) = \pu\\
-1 ,&\tr(\MatR_j^\phi\W^*) = \pl
\end{array}
\right..
\end{align*}
The strict feasibility in A1 guarantees that $\pu$ and $\pl$ cannot be attained simultaneously, so the definition above is fully specified. Similarly we define
\begin{align*}
\indq(j,\phi):=\left\{
\begin{array}{ll}
0  ,&\tr(\MatI_j^\phi\W^*) \not\in\{\qu,\ql\}\\
1  ,&\tr(\MatI_j^\phi\W^*)=\qu\\
-1 ,&\tr(\MatI_j^\phi\W^*)=\ql
\end{array}
\right..
\end{align*}
\begin{definition}
The {\it critical objective} bus set is
\begin{align*}
\objset:=\{j\in\Vertex:\exists\phi~\text{s.t.}~c_{j,{\rm re}}^\phi\neq 0~\text{or}~c_{j,{\rm im}}^\phi\neq 0\}.
\end{align*}
\end{definition}
\begin{definition}
The {\it critical constraint} bus set is
\begin{align*}
\conset:=\{j\in\Vertex:\exists\phi~\text{s.t.}~\indp(j,\phi)\neq 0~\text{or}~\indq(j,\phi)\neq 0\}.
\end{align*}
\end{definition}

For any $mn\times mn$ matrix $\X$, we use $[\X]_{j,k}$ to denote the $m\times m$ block of $\X$ from rows $jm+1$ to $jm+m$ and from columns $km+1$ to $km+m$.
Further, for $\phi\in\M$, we denote $[\X]_{j,k}^{\phi,:}$ and $[\X]_{j,k}^{:,\phi}$ as the $\phi$\textsuperscript{th} row and column of $[\X]_{j,k}$, respectively.
Similarly, for an $mn$ dimensional vector $\x$, we use $[\x]_{j}$ to denote the subvector of $\x$ from the $(jm+1)$\textsuperscript{th} to $(jm+m)$\textsuperscript{th} entry.
Denote
\begin{align*}
\getnz(\x):=\{j\in\Vertex, [\x]_{j}\neq\zero\}
\end{align*}
and we use $|\getnz|$ to denote its cardinality.

We say $\Vertex_1\subseteq\Vertex$ is {\it connected} in $\Graph$ if $\Graph$ has a connected subgraph whose vertex set is $\Vertex_1$. 
For any node $j\in\Vertex$, we denote the set of its neighbors in $\Graph$ as $\neighbor(j)$. For $\mathcal{K}\subseteq \Vertex$, we reload $\neighbor(\mathcal{K}):=\cup_{j\in\mathcal{K}}\neighbor(j)$.

We say a set of real numbers are {\it sign-semidefinite} if all the non-zero numbers are of the same sign.

\subsection{Main Results}
Consider the following conditions.
\begin{enumerate}[resume,label=A\arabic*:]
\item The underlying graph $\Graph$ is a tree.
\item $(\objset\cup\conset)\cap\neighbor(\objset\cup\conset)=\emptyset$.
\item $\objset\cap\conset=\emptyset$.
\item For any $j\in\objset\cap\conset$ and $\phi\in\M$, $c_{j,{\rm re}}^\phi\indp(j,\phi)\geq 0$ and $c_{j,{\rm im}}^\phi\indq(j,\phi)\geq 0$.
\end{enumerate}
Informally, A3 means all the critical buses are not adjacent to each other.
A5 means if a bus is both critical in objective function and constraints, then for all $m$ phases, $\{c_{j,{\rm re}}^\phi, \indp(j,\phi)\}$ and $\{c_{j,{\rm im}}^\phi, \indq(j,\phi)\}$ are sign-semidefinite, respectively.
The following two theorems provide two sets of sufficient conditions for exact SDP relaxation.
\begin{theorem}\label{thm:sc1}
If conditions A1, A2, A3 and A4 hold, then \eqref{eq:opf3} is exact.
\end{theorem}
\begin{theorem}\label{thm:sc2}
If conditions A1, A2, A3 and A5 hold, then \eqref{eq:opf3} is exact.
\end{theorem}

Both theorems rely on strict feasibility, tree structure and critical buses not be adjacent.
Theorem \ref{thm:sc1} needs $\objset$ and $\conset$ to be also disjoint.
On the other hand, Theorem \ref{thm:sc2} allows them to intersect, but says for each $(j,\phi)$ in the intersection, the objective and constraints should encourage its injection to move in the same direction.
\footnote{For example, if $\R(\s_j^\phi)$ is minimized in the objective function, then the lower bound of $\R(\s_j^\phi)$ should not be active in the constraints.}
Since A4 implies A5, Theorem \ref{thm:sc2} is stronger than Theorem \ref{thm:sc1}. 
In the next section, we will only provide a proof of Theorem \ref{thm:sc2}.

One drawback of Theorems \ref{thm:sc1} and \ref{thm:sc2} is that the sufficient conditions are given in terms of
the optimal solution $\W^*$.
The next result provides a sufficient condition that depends only on the primal parameters in \eqref{eq:opf1}.
Let 
\begin{align*}
\consetcol:=\{j\in\Vertex:\exists\phi~\text{s.t.}~\{\pm\infty\}\not\subseteq\{\underline{p}_j^\phi,\overline{p}_j^\phi,\underline{q}_j^\phi,\overline{p}_j^\phi\}\}.
\end{align*}

\begin{corollary}\label{col:primalcon}
Suppose A1 and A2 hold,
If $(\objset\cup\consetcol)\cap\neighbor(\objset\cup\consetcol)=\emptyset$ and $\objset\cap\consetcol=\emptyset$, then \eqref{eq:opf3} is exact.
\end{corollary}
\begin{proof}
As $\conset\subseteq\consetcol$, the conditions in the corollary imply A1--A4 and thus exactness holds.
\end{proof}

Informally, Corollary \ref{col:primalcon} shows that if all the buses involved in the objective function and constraints are not adjacent to each other, then the SDP relaxation is exact.

%% file: Proof.tex
\section{Proof of Sufficient Conditions}
\subsection{Review}

The existing works \cite{bose2015quadratically, sojoudi2014exactness} prove that the optimal solution of SDP relaxation is of rank 1 in single phase networks. 
A crucial step in their proof uses the strong duality to show that the product of the primal optimal solution $\W^*$ and the dual matrix $\A^*$ is a zero matrix, and hence the rank of $\W^*$ cannot exceed the dimension of $\A^*$'s null space. Under certain conditions \cite{bose2015quadratically, sojoudi2014exactness} prove that $\A^*$'s null space is of dimension at most $1$. Hence the optimal primal solution $\W^*$ must be of rank at most $1$.

This argument however breaks down in a multiphase network for the following two reasons. First, although the underlying graph for $m$ phase network is still a tree, each bus now has $m$ different phases and might have $m$ unbalanced voltages in general. If we extend each phase to a separate vertex in the new graph and connect every phase pair between every two neighboring buses, then the $m$ phase network will be transformed into an $(mn)$-node meshed network with multiple cycles \cite{Berg1967, Laughton1968, Chen1991}.
Hence the theory for single-phase radial network is not applicable.
Second, in an $m$ phase network, it is unknown wether the null space of $\A^*$ at the optimal point is still of dimension $1$. 
It is therefore not clear how to prove $\rank(\W^*)=1$ via analyzing the dimension of $\nul(\A^*)$.

In the following argument, we  use a similar proof framework to that in \cite{bose2015quadratically}, but the proof will be based on the eigenvectors of $\W^*$ instead of the dimension of $\nul(\A^*)$.
From now on, we suppose A1, A2, A3 and A5 hold.

\subsection{Preliminaries}
Our strategy is to prove the exactness of the perturbed OPF problem and then use Lemma \ref{thm:perturbation} to show \eqref{eq:opf3} is also exact.
It is important to make sure that all the non-active constraints will remain non-active in the perturbation neighborhood.
\begin{lemma}\label{thm:sign}
For any nonzero $\C_1$, there exists a positive sequence $\eps\downarrow 0$ such that for each $\eps$ in the sequence,
one can collect $(\dpu(\eps),\dpl(\eps),\dqu(\eps),\dql(\eps))$ from at least one of its KKT multiplier tuples satisfying
\begin{subequations}
\begin{align}
\indp(j,\phi)=0&\implies \dpu(\eps)=\dpl(\eps)=0\label{eq:implication.a}\\
\indp(j,\phi)\neq0&\implies \indp(j,\phi)\cdot(\dpu(\eps)-\dpl(\eps))\geq0\label{eq:implication.b}\\
\indq(j,\phi)=0&\implies \dqu(\eps)=\dql(\eps)=0\label{eq:implication.c}\\
\indq(j,\phi)\neq0&\implies \indq(j,\phi)\cdot(\dqu(\eps)-\dql(\eps))\geq0.\label{eq:implication.d}
\end{align}\label{eq:implication}
\end{subequations}
\end{lemma}
\begin{proof}
First consider any positive sequence $\{{\eps}_l\}_{l=1}^{\infty}$ such that $\lim_{l\to\infty}\eps_l=0$. 
Suppose the optimal solution to \eqref{eq:opf_perturb} under $\eps_l$ is $\W_l$ (if there are multiple solutions then select one of them).
As \eqref{eq:opf_perturb.b} prescribes a compact set, using a similar argument as in the proof of Lemma \ref{thm:perturbation} we know there must be a subsequence of $\{{\eps}_l\}_{l=1}^{\infty}$, denoted by $\{{\eps}_{z_t}\}_{t=1}^{\infty}$, 
such that $\W_{z_t}$ converges to $\W^*$ in the max norm. The difference $\|\W_{z_t}-\W^*\|_\infty$ can be arbitrarily small for sufficiently large $t$.
When $t$ is large enough, the non-active constraints in \eqref{eq:opf_perturb.b} under $\W^*$ will remain non-active under $\W_{z_t}$, and the corresponding KKT multipliers will remain $0$.
As a result,
\begin{align*}
&\indp(j,\phi)=0\implies \pl<\tr(\MatR_j^\phi\W^*)<\pu\\
\implies &\pl<\tr(\MatR_j^\phi\W_{z_t})<\pu\implies\dpu(\eps_{z_t})=\dpl(\eps_{z_t})=0,
\end{align*}
\begin{align*}
&\indp(j,\phi)=+1\implies\pl<\tr(\MatR_j^\phi\W^*)\\
\implies&\pl<\tr(\MatR_j^\phi\W_{z_t})
\implies\dpl(\eps_{z_t})=0\\
\implies&\indp(j,\phi)\cdot(\dpu(\eps_{z_t})-\dpl(\eps_{z_t}))\geq0,
\end{align*}
\begin{align*}
&\indp(j,\phi)=-1\implies\tr(\MatR_j^\phi\W^*)<\pu\\
\implies&\tr(\MatR_j^\phi\W_{z_t})<\pu
\implies\dpu(\eps_{z_t})=0\\
\implies&\indp(j,\phi)\cdot(\dpu(\eps_{z_t})-\dpl(\eps_{z_t}))\geq0
\end{align*}
all hold. A similar argument can also be applied to prove \eqref{eq:implication.c} and \eqref{eq:implication.d}.
\end{proof}

\subsection{Properties of Dual Matrix $\A^*(\eps)$}
In order to apply Lemma \ref{thm:perturbation}, we construct $\C_1\in\Complex^{mn\times mn}$ in the following manner.
\begin{align*}
[\C_1]_{jj}=\zero\in\Complex^{m\times m}&,\quad \text{for}~j\in\Vertex\\
[\C_1]_{jk}=\zero\in\Complex^{m\times m}&,\quad \text{for}~(j,k)\not\in\Edge
\end{align*}
When $(j,k)\in\Edge$, we assume $j<k$. If neither $j$ nor $k$ is in $\objset\cup\conset$, then we construct $[\C_1]_{jk}=\Y_{jk}$.

If $j\in\objset\cup\conset$, then A3 guarantees $k\not\in\objset\cup\conset$. $\forall\phi\in\M$, 
we set $[\C_1]_{jk}^{\phi,:}$ to $\Y_{jk}^{\phi,:}$ if $c_{j,{\rm re}}^\phi=c_{j,{\rm im}}^\phi=\indp(j,\phi)=\indq(j,\phi)=0$, and to $(\indp(j,\phi)+\indq(j,\phi)\iu)\Y_{jk}^{\phi,:}$ otherwise.

If $k\in\objset\cup\conset$, then A3 guarantees $j\not\in\objset\cup\conset$. $\forall\phi\in\M$, 
we similarly set $[\C_1]_{jk}^{:,\phi}$ to $(\Y_{kj}^{\phi,:})^\Hn$ if $c_{k,{\rm re}}^\phi=c_{k,{\rm im}}^\phi=\indp(k,\phi)=\indq(k,\phi)=0$, and to $(\indp(k,\phi)-\indq(k,\phi)\iu)(\Y_{kj}^{\phi,:})^\Hn$ otherwise.

Finally, we set $[\C_1]_{kj}:=[\C_1]_{jk}^\Hn$ for all $j<k$ to make $\C_1$ Hermitian.

\begin{definition}\label{df:Ginv}
An $mn\times mn$ positive semidefinite matrix $\X$ is $\Graph$-{invertible} for some graph $\Graph$ if the following two conditions hold:
\begin{enumerate}
\item $\forall (a,b)\in\Edge$, $[\X]_{ab}$ is invertible.
\item $\forall a,b\in\Vertex$ such that $a\neq b$ and $(a,b)\not\in\Edge$, $[\X]_{ab}$ is all zero.
\end{enumerate}
\end{definition}

The next theorem provides a key intermediate result to prove Theorem \ref{thm:sc2}.
Suppose under such $\C_1$, the sequence guaranteed by Lemma \ref{thm:sign} is $\{\eps_l\}_{l=1}^{\infty}$.
\begin{theorem}\label{thm:Aproperty}
Under A1, A2, A3 and A5, for each $\eps_l$, the dual matrix $\A^*(\eps_l)$ is $\Graph$-{invertible}.
\footnote{If the KKT multiplier tuple at $\eps_l$ is non-unique, then $\A^*(\eps_l)$ is evaluated at the multiplier tuple in Lemma \ref{thm:sign} satisfying \eqref{eq:implication}.}
\end{theorem}
\begin{proof}
The value of $\A^*(\eps_l)$ is the same as the right hand side of \eqref{eq:A} when all dual variables take values at their corresponding KKT multipliers (with respect to $\eps_l$).
If not otherwise specified, all the $(\dpu,\dpl,\dqu,\dql)$ in this proof refer to the tuple in Lemma \ref{thm:sign} with respect to $\eps_l$.
Since for all $a\neq b$, $[\Base_j^\phi]_{ab}$ and $[\Pi(\dslack)]_{ab}$ are always zero matrices, it is sufficient to show
\begin{align*}
\Q:=\sum\limits_{j,\phi}\Big((\dpu-\dpl)\MatR_j^\phi+(\dqu-\dql)\MatI_j^\phi\Big)+\C_0+\eps_l\C_1
\end{align*}
satisfies the two conditions in Definition \ref{df:Ginv}.\footnote{The matrix $\Q$ itself might not be $\Graph$-invertible as $\Q$ might not be positive semidefinite, but $\A^*\geq0$ always hold.}

For $a\neq b$ and $(a,b)\not\in\Edge$, recall that $\C_0$ is the linear combination of $\MatR_j^{\phi}$ and $\MatI_j^{\phi}$.
When $(a,b)\not\in\Edge$, $\Y_{ab}$ is a zero matrix and so are all $[\MatR_j^{\phi}]_{ab}$ and $[\MatI_j^{\phi}]_{ab}$.
The construction of $\C_1$ also guarantees $[\C_1]_{ab}$ is all zero. Hence $[\Q]_{ab}$ is all zero as well.

Now assume $a<b$. If $(a,b)\in\Edge$, we have
\begin{align}\label{eq:Qab}
\nonumber
&[\Q]_{ab}\\
\nonumber
=&\sum_{\phi}\!\Big((\dpur_a^\phi-\dplr_a^\phi+c_{a,{\rm re}}^\phi)[\MatR_a^\phi]_{ab}\!+\!(\dqur_a^\phi-\dqlr_a^\phi+c_{a,{\rm im}}^\phi)[\MatI_a^\phi]_{ab}\Big)\\
\nonumber
+&\sum_{\phi}\!\Big((\dpur_b^\phi-\dplr_b^\phi+c_{b,{\rm re}}^\phi)[\MatR_b^\phi]_{ab}\!+\!(\dqur_b^\phi-\dqlr_b^\phi+c_{b,{\rm im}}^\phi)[\MatI_b^\phi]_{ab}\Big)\\
+&\eps_l[\C_1]_{ab}.
\end{align}

If neither $a$ nor $b$ is in $\objset\cup\conset$, then by definition, for all $\phi\in\M$ there must be 
\begin{subequations}
\begin{align}
c_{a,{\rm re}}^\phi=c_{a,{\rm im}}^\phi=\indp(a,\phi)=\indq(a,\phi)=0,
\label{eq:para0.a}\\
c_{b,{\rm re}}^\phi=c_{b,{\rm im}}^\phi=\indp(b,\phi)=\indq(b,\phi)=0.
\label{eq:para0.b}
\end{align}
\end{subequations}
Equation \eqref{eq:Qab} and Lemma \ref{thm:sign} imply $[\Q]_{ab}=\eps_l[\C_1]_{ab}$. By construction, $[\C_1]_{ab}=\Y_{ab}$ is invertible, and so is $[\Q]_{ab}$.

If $a\in\objset\cup\conset$, then A3 guarantees $b\not\in\objset\cup\conset$. Thus \eqref{eq:para0.b} holds for all $\phi\in\M$.
For a given $\phi\in\M$, if \eqref{eq:para0.a} holds, then by construction, we have $[\Q]_{ab}^{\phi,:}=\eps_l[\C_1]_{ab}^{\phi,:}=\eps_l\Y_{ab}^{\phi,:}$.
If \eqref{eq:para0.a} does not hold for the given $\phi$, then we have
\begin{align*}
[\Q]_{ab}^{\phi,:}&=(\dpur_a^\phi-\dplr_a^\phi+c_{a,{\rm re}}^\phi+2\eps_l\indp(a,\phi))\frac{\Y_{ab}^{\phi,:}}{2}\\
&+(\dqur_a^\phi-\dqlr_a^\phi+c_{a,{\rm im}}^\phi+2\eps_l\indq(a,\phi))\frac{\Y_{ab}^{\phi,:}}{2}\iu.
\end{align*}
Note that Condition A5 and Lemma \ref{thm:sign} imply both $\{\dpur_a^\phi-\dplr_a^\phi,\indp(a,\phi),c_{a,{\rm re}}^\phi\}$ and $\{\dqur_a^\phi-\dqlr_a^\phi,\indq(a,\phi),c_{a,{\rm im}}^\phi\}$ are sign-semidefinite sets, respectively.
When \eqref{eq:para0.a} does not hold, at least one of $\{c_{a,{\rm re}}^\phi, c_{a,{\rm im}}^\phi, \indp(a,\phi), \indq(a,\phi)\}$ is non-zero.
As a result, there exists some non-zero $\coef\in\Complex$ such that $[\Q]_{ab}^{\phi,:}=\coef\Y_{ab}^{\phi,:}$.
In short, in the case $a\in\objset\cup\conset$, $[\Q]_{ab}^{\phi,:}$ is always a non-zero multiple of $\Y_{ab}^{\phi,:}$. 
The invertibility of $\Y_{ab}$ indicates all the $\Y_{ab}^{\phi,:}$ are independent for $\phi\in\M$, so $[\Q]_{ab}$ is also invertible.

If $b\in\objset\cup\conset$, then A3 guarantees $a\not\in\objset\cup\conset$. Then \eqref{eq:para0.a} holds for all $\phi\in\M$.
For a given $\phi\in\M$, if \eqref{eq:para0.b} holds, then by construction, we have $[\Q]_{ab}^{:,\phi}=\eps_l[\C_1]_{ab}^{:,\phi}=\eps_l(\Y_{ba}^{\phi,:})^\Hn$.
If \eqref{eq:para0.b} does not hold, then 
similar to the previous case, there exists some non-zero $\coefc\in\Complex$ such that $[\Q]_{ab}^{:,\phi}=\coefc(\Y_{ba}^{\phi,:})^\Hn$.
Hence $[\Q]_{ab}^{:,\phi}$ is always a non-zero multiple of $(\Y_{ba}^{\phi,:})^\Hn$. 
The invertibility of $\Y_{ba}$ indicates all the $\Y_{ba}^{\phi,:}$ are independent for $\phi\in\M$, so $[\Q]_{ab}$ is also invertible.
\end{proof}

The next theorem is a generalization of Theorem 3.3 in \cite{van2003graphs}.
While \cite{van2003graphs} studies the matrices whose non-zero off-diagonal entries correspond to an edge in $\Graph$, 
we extend the results to $\Graph$-invertible matrices.
\begin{theorem}\label{thm:connectivity}
Let $\y\in\Complex^{mn}$ be a non-zero vector with the smallest $|\getnz(\y)|$ satisfying $\X\y=\zero$, where $\X$ is $\Graph$-invertible.
Then $\getnz(\y)$  is connected in $\Graph$.
\end{theorem}
\begin{proof}
If not, then assume $\getnz(\y)=\getnz_1\cup\getnz_2$ where non-empty sets $\getnz_1$ and $\getnz_2$ are not connected in $\Graph$.
Construct $\yy$ in the following manner:
\begin{align*}
[\yy]_k=\left\{
\begin{array}{ll}
[\y]_k&,k\not\in\getnz_2\\
\zero&,k\in\getnz_2
\end{array}
\right..
\end{align*}
Then for each $j\in\getnz_1$, 
\begin{align*}
[\X\yy]_j=&\sum_{k\in\Vertex}[\X]_{jk}[\yy]_k
=[\X]_{jj}[\yy]_j+\sum_{k:k\sim j}[\X]_{jk}[\yy]_k\\
=&[\X]_{jj}[\y]_j+\sum_{k:k\sim j}[\X]_{jk}[\y]_k=[\X\y]_j=\zero.
\end{align*}
The third equality above is due to the fact that $j\in\getnz_1$ is not connected to any nodes in $\getnz_2$.
Therefore,
\begin{align*}
\yy^\Hn\X\yy=&\sum_{j\in\Vertex}[\yy]_j^\Hn[\X\yy]_j
=\sum_{j\in\getnz_1}[\yy]_j^\Hn[\X\yy]_j + \sum_{j\not\in\getnz_1}[\yy]_j^\Hn[\X\yy]_j\\
=&\sum_{j\in\getnz_1}[\yy]_j^\Hn\zero + \sum_{j\not\in\getnz_1}\zero^\Hn[\X\yy]_j=\zero.
\end{align*}
Since $\Graph$-invertibility implies $\X\geq0$, there must be $\X\yy=\zero$ as well.
As $|\getnz(\yy)|=|\getnz_1|<|\getnz(\y)|$ and $\yy$ is non-zero by construction, it contradicts the minimality of $|\getnz(\y)|$.
\end{proof}

\subsection{Proof of Theorem \ref{thm:sc2}}
We now prove that \eqref{eq:opf3} is exact under conditions A1, A2, A3 and A5.
According to Lemma \ref{thm:perturbation}, we only need to show \eqref{eq:opf_perturb} is exact for any $\eps_l$ in the sequence $\{\eps_l\}_{l=1}^{\infty}$ used in Theorem \ref{thm:Aproperty}.
If \eqref{eq:opf_perturb} is not exact, then there exists an optimal solution $\W^*$ such that $\rank(\W^*)\geq 2$.
\footnote{Note that $\rank(\W^*)$ cannot be 0 as the constraint $[\W^*]_{00}=\VV_{\rm ref}$ requires $\W^*$ to be a non-zero matrix.}
Note that in this subsection, $\W^*$ stands for the optimal solution to \eqref{eq:opf_perturb}.

Suppose the eigen-decomposition of $\W^*$ is 
\begin{align*}
\W^*=\sum_{l=1}^{mn}\eigv_l\eigvec_l\eigvec_l^\Hn
\end{align*}
where $\eigv_1\geq\eigv_2\geq\dots\eigv_{mn}\geq0$ are $\W^*$'s eigenvalues in  decreasing order and $\eigvec_l$ is the eigenvector associated with $\eigv_l$.
All the $\eigvec_l$ are non-zero and orthogonal.
As $\rank(\W^*)\geq 2$, we have $\eigv_2>0$.
Now let $2\leq L\leq mn$ be the largest number such that $\eigv_L>0$, then we have
\begin{align*}
\V_{\rm ref}\V_{\rm ref}^\Hn=[\W^*]_{00}=\sum_{l=1}^{L}\eigv_l[\eigvec_l]_0[\eigvec_l]_0^\Hn=:\U\U^\Hn,\\
\U:=\Big[\sqrt{\eigv_1}[\eigvec_1]_0,\sqrt{\eigv_2}[\eigvec_2]_0,\dots,\sqrt{\eigv_L}[\eigvec_L]_0\Big].
\end{align*}
If the rank of $\U$ is strictly greater than $1$, then we can find $\z\in\spanv(\U)$ such that $\z^\Hn\V_{\rm ref}=0$.
Then $\U^\Hn\z\neq 0$ implies
\begin{align*}
0=&\z^\Hn\V_{\rm ref}\V_{\rm ref}^\Hn\z=\z^\Hn\U\U^\Hn\z>0.
\end{align*}
The contradiction means $\rank(\U)\leq1$, and therefore $[\eigvec_1]_0$ and $[\eigvec_2]_0$ are linearly dependent.
If $[\eigvec_1]_0=r[\eigvec_2]_0$ for some $r\in\Complex$, then we construct $\keyvec=\eigvec_1-r\eigvec_2$.
Otherwise $[\eigvec_2]_0$ must be zero and we construct $\keyvec=\eigvec_2$.
Clearly we have
\begin{align}\label{eq:keyvec}
\keyvec\neq0, [\keyvec]_0=\zero.
\end{align}

On the other hand, KKT conditions give $\tr(\A^*\W^*)=0$.
As both $\A^*$ and $\W^*$ are positive semidefinite, we have
\begin{align*}
0=&\tr(\A^*\W^*)=\tr\Big(\A^*\sum_{l=1}^L \eigv_l\eigvec_l\eigvec_l^\Hn\Big)\\
=&\sum_{l=1}^L\tr\big(\eigv_l\A^*\eigvec_l\eigvec_l^\Hn\big)=\sum_{l=1}^L\tr\big(\eigv_l\eigvec_l^\Hn\A^*\eigvec_l\big)\geq0.
\end{align*}
The equality holds only when $\A^*\eigvec_l=\zero$ for all $l\leq L$. Hence
\begin{align}\label{eq:Akeyvec}
\A^*\keyvec=\zero.
\end{align}

As \eqref{eq:keyvec} has shown $1\leq|\getnz(\keyvec)|\leq n-1$, putting together Theorem \ref{thm:Aproperty}, Theorem \ref{thm:connectivity} and \eqref{eq:Akeyvec} implies that 
there exists $\minvec$ such that $\getnz(\minvec)$ is non-empty, connected in $\Graph$, $1\leq|\getnz(\minvec)|\leq n-1$, and $\A^*\minvec = \zero$.
Let $j$ be a node not in $\getnz(\minvec)$ but is connected to some node $k\in\getnz(\minvec)$. 
Since A2 requires $\Graph$ to be a tree and $\getnz(\minvec)$ is connected in $\Graph$,
$k$ must be the only node in $\getnz(\minvec)$ which is connected to $j\not\in\getnz(\minvec)$. Otherwise there is a cycle.
Then
\begin{align*}
[\A^*\minvec]_j=&\sum_{l\in\Vertex}[\A^*]_{jl}[\minvec]_l=[\A^*]_{jj}[\minvec]_j+\sum_{l:l\sim j}[\A^*]_{jl}[\minvec]_l\\
=&[\A^*]_{jj}\zero+[\A^*]_{jk}[\minvec]_k+\sum_{l:l\sim j,l\not\in\getnz(\minvec)}[\A^*]_{jl}[\minvec]_l.
\end{align*}
As $[\minvec]_l=\zero$ for $l\not\in\getnz(\minvec)$, we have $[\A^*\minvec]_j=[\A^*]_{jk}[\minvec]_k$.
Further, $(j,k)\in\Edge$ and the $\Graph$-invertibility of $\A^*$ implies $[\A^*]_{jk}$ is invertible.
Node $k$ is in $\getnz(\minvec)$ implies $[\minvec]_k\neq\zero$.
As a result, $[\A^*\minvec]_j=[\A^*]_{jk}[\minvec]_k$ must be non-zero, contracting $\A^*\minvec = \zero$.
This  implies that \eqref{eq:opf_perturb} is exact. Theorem \ref{thm:sc2} is proved.
\hfill\QED

%% file: Simulation.tex
\begin{table*}
\vspace{0.5em}
\caption{Illustrative example summary.}
\centering
\begin{tabular}{|c *{11}{V{3} c|c|c} |}
\hline
Buses & \multicolumn{3}{c V{3} }{650} & \multicolumn{3}{c V{3} }{632} & \multicolumn{3}{c V{3} }{633} & \multicolumn{3}{c V{3} }{634} & \multicolumn{3}{c V{3} }{645} & \multicolumn{3}{c V{3} }{646} & \multicolumn{3}{c V{3} }{671} & \multicolumn{3}{c V{3} }{684} & \multicolumn{3}{c V{3} }{611} & \multicolumn{3}{c V{3} }{652} & \multicolumn{3}{c|}{680}\\
\hline
Phases & \!\!a\!\! & \!\!b\!\! &\!\!c\!\! & \!\!a\!\! & \!\!b\!\! &\!\!c\!\! & \!\!a\!\! & \!\!b\!\! &\!\!c\!\! & \!\!a\!\! & \!\!b\!\! &\!\!c\!\! & \!\!a\!\! & \!\!b\!\! &\!\!c\!\! & \!\!a\!\! & \!\!b\!\! &\!\!c\!\! & \!\!a\!\! & \!\!b\!\! &\!\!c\!\! & \!\!a\!\! & \!\!b\!\! &\!\!c\!\! & \!\!a\!\! & \!\!b\!\! &\!\!c\!\! & \!\!a\!\! & \!\!b\!\! &\!\!c\!\! & \!\!a\!\! & \!\!b\!\! &\!\!c\!\! \\
\hline
Objective (real)       & \!\!+\!\! & \!\!+\!\! & \!\!+\!\! & & & & \!\!-\!\! & \!\!-\!\! & \!\!-\!\! & & & & \!\!+\!\! & \!\!+\!\! & \!\!+\!\! & & & & \!\!-\!\! & \!\!-\!\! & \!\!-\!\! & & & &  \!\!+\!\! & \!\!+\!\! & \!\!+\!\! & \!\!+\!\! & \!\!+\!\! & \!\!+\!\! & & & \\
\hline
Objective (reactive) & \!\!+\!\! & \!\!+\!\! & \!\!+\!\! & & & & \!\!-\!\! & \!\!-\!\! & \!\!-\!\! & & & & \!\!-\!\! & \!\!-\!\! & \!\!-\!\! & & & & \!\!+\!\! & \!\!+\!\! & \!\!+\!\! & & & & \!\!+\!\! & \!\!+\!\! & \!\!+\!\! & & & & & & \\
\hline
Constraints (real) & \up &\up & \up & & & & \lo & \lo & \lo & & & & \up & \up & \up & & & & \lo & \lo & \lo & & & &  \up & \up & \up & & & & & & \\
\hline
Constraints (reactive) & \cellcolor{red!25}\up &\cellcolor{red!25}\up & \cellcolor{red!25}\up & & & & \lo & \lo & \lo & & & & \lo & \lo & \lo & & & & \cellcolor{red!25}\up & \cellcolor{red!25}\up & \cellcolor{red!25}\up & & & &\cellcolor{red!25}\up &\cellcolor{red!25}\up &\cellcolor{red!25}\up & & & & & & \\
\hline
\end{tabular}
\label{tb:programming}
\begin{spacing}{0.0}
\end{spacing}
\end{table*}

\section{Discussion and Example}\label{sec:discuss}
\subsection{Discussion}
The main results in this paper are Theorem \ref{thm:sc1}, Theorem \ref{thm:sc2} and Corollary \ref{col:primalcon}.
They provide three sets of sufficient conditions under which the SDP relaxation for unbalanced multiphase network is exact.
These results have different interpretations and implications.

Sufficient conditions in Corollary \ref{col:primalcon} do not rely on the optimal solution of SDP relaxation, and can be checked {\it a priori}. 
Though these conditions are still restrictive in practice, we hope this result can stimulate more work on unbalanced multiphase networks.

Conditions in Theorems \ref{thm:sc1} and \ref{thm:sc2} rely on knowing the active constraints at the optimal point, which cannot be checked {a priori}.
Nevertheless, the actual value of the optimal point is not involved as long as one knows where the bottlenecks are.
These conditions also suggest that relaxation is more likely to be exact if critical buses turn out
to be spread over the network rather than concentrated in some neighborhood

So far we have assumed that \eqref{eq:opf3} has a unique optimal solution so that
inactive constraints at the optimal solution of \eqref{eq:opf3} remain inactive under a small perturbation.
If \eqref{eq:opf3} has multiple solutions, A4 and A5 in Theorems \ref{thm:sc1} and \ref{thm:sc2} and condition $\objset\cap\consetcol=\emptyset$ in Corollary \ref{col:primalcon} need to be replaced by the linear separability 
condition proposed in \cite{bose2015quadratically}.
The proof will be similar to that in this paper.

To generalize the result here to nonlinear cost functions, 
note that the proposed conditions involving the cost function only rely on the signs of $c_{j,{\rm re}}^\phi$ and $c_{j,{\rm im}}^\phi$.
The same argument in this paper can be extended to the nonlinear case when the cost function is convex, monotonic and additively separable in injections.

\subsection{Illustrative Example}
We use an 11 bus radial network shown in Fig. \ref{fig:network}, adapted from IEEE 13 node test feeder,  to illustrate our theoretical result.
\begin{figure}
\centering
\vspace{0.2em}
\includegraphics[width=0.76\columnwidth]{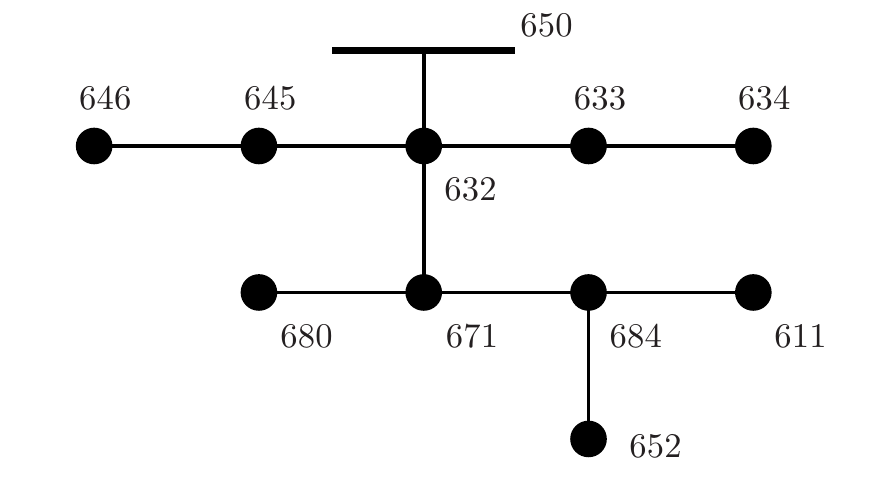}
\caption{An $11$ bus network revised from IEEE 13 node test feeder. The switch in the original system is assumed to be open so 2 buses are removed.}
\label{fig:network}
\end{figure}
The line configuration is reassigned and noise is added to the admittance matrix, so all the buses have three complete phases and each $y_{jk} $ is invertible.
For illustrative purpose, all the real/reactive injections are bounded from at most one direction.
Table \ref{tb:programming} summarizes our setup.
The `+' and `-' refer to the sign of $c_{j,{\rm re}}^\phi$ or $c_{j,{\rm im}}^\phi$ in the cost function.
For constraints, `u' (or `l') means the upper (or lower) bound for the corresponding injection is finite.
It is easy to check that no matter which constraints are active at the optimal point, conditions A1, A2, A3 and A5 must hold,
so Theorem \ref{thm:sc2} implies the optimal solution is of rank $1$. 

After solving the problem, there are actually nine active constraints, highlighted in light red in Table \ref{tb:programming}.
The largest two eigenvalues of the resulting optimal solution $\W^*$ are $36.90$ and $1.44\times 10^{-10}$, respectively.
It confirms that $\W^*$ is indeed rank 1 up to numerical precision.

Finally, we refer to \cite{gan2014convex} for more simulation results, 
which show that semidefinite relaxation is also exact for IEEE 13, 37, 123-bus networks and a real-world 2065-bus network.
In the simulation of \cite{gan2014convex}, our sufficient conditions are actually violated since the cost function is set as
\begin{align*}
\sum\limits_{j\in\Vertex}\sum\limits_{\phi\in\M}\R(\s_j^\phi).
\end{align*}
It means even when all the buses are critical, the semidefinite relaxation can still be exact.

\section{Conclusion}
We have proposed 
sufficient conditions for exact SDP relaxation in unbalanced multiphase radial networks.
These conditions suggest that having critical buses not adjacent to each other encourages exact relaxation.